\documentclass[11pt]{amsart}
\usepackage{amsmath,amssymb,amsfonts,amsthm,amscd,indentfirst}
\usepackage{amsmath,amssymb,amsfonts,amsthm,amscd}
\usepackage{indentfirst}
\usepackage{hyperref}

\numberwithin{equation}{section}

\newtheorem{prop}{Proposition}[section]
\newtheorem{theo}[prop]{Theorem}
\newtheorem{lemm}[prop]{Lemma}
\newtheorem{coro}[prop]{Corollary}
\newtheorem{rema}[prop]{Remark}

\def\and{\quad{\rm and}\quad}

\def\<{\langle}
\def\>{\rangle}

\title{Inverse curvature flow in anti-de Sitter-Schwarzschild manifold}
\author{
        Siyuan Lu
        }
    
\address{Department of Mathematics and Statistics, McGill University, 805 Sherbrooke O, Montreal, Quebec, Canada, H3A 0B9}
\email{siyuan.lu@mail.mcgill.ca}
\thanks{Research of the author was supported in part by CSC fellowship and Schulich Graduate fellowship.}
\usepackage{graphicx}

\begin{document}

\begin{abstract}
In this paper, we consider the inverse hessian quotient curvature flow with star-shaped initial hypersurface in anti-de Sitter-Schwarzschild manifold. We prove that the solution exists for all time, and the second fundamental form converges to identity exponentially fast.
\end{abstract}

\maketitle

\section{Introduction}
Curvature flows of compact hypersurfaces in Riemannian manifolds have been extensively studied in the last 30 years. In the case of Euclidean space, for contracting flow, Huisken \cite{H1} considered
\begin{align}\label{MCF}
\dot{X}=-H\nu
\end{align}
where $H$ is the mean curvature. He proved that the solution exists for all time and the normalized flow converges to a round sphere if the initial hypersurface is convex. 

This result is later generalized by Andrews \cite{A1} for a large class of curvature flow. More specificly, Andrews considered
\begin{align}\label{CF}
\dot{X}=-F\nu
\end{align}
where $F$ is a concave function of homogeneous degree one, evaluated at the principal curvature. 

For expanding flow, Gerhardt \cite{G1} and Urbas \cite{U} considered 
\begin{align}\label{ICF}
\dot{X}=\frac{\nu}{F}
\end{align}
where $F$ is a concave function of homogeneous degree one, evaluated at the principal curvature. They proved that the solution exists for all time and the normalized flow converges to a round sphere if the initial hypersurface is star-shaped and lies in a certain convex cone.

A natural question is whether these results remain true if the ambient space is no longer Euclidean space. For contraction flow (\ref{MCF}) and (\ref{CF}), Huisken \cite{H2} and Andrews \cite{A2} generalized their results to certain ambient space respectively.

The case of expanding flow (\ref{ICF}) is in fact more subtle as the assumption on initial hypersurface is weaker. In the case of space form, Gerhardt \cite{G2,G3} proved the solution exists for all time and the second fundamental form converges in hyperbolic space and sphere space, see also earlier work by Ding \cite{D}. More recently, Brendle-Hung-Wang \cite{BHW} and Scheuer \cite{S2} proved that the same results hold in anti-de Sitter-Schwarzschild manifold and a class of warped product manifold for inverse mean curvature flow, which is
\begin{align}\label{IMCF}
\dot{X}=\frac{\nu}{H}
\end{align}
However, as pointed out by Neves \cite{N} and Hung-Wang in \cite{HW}, for inverse mean curvature flow, the rescaled hypersurface is not necessary a round sphere in anti-de Sitter-Schwarzschild manifold and in hyperbolic space.

Inverse curvature flows can be used to prove various inequalities. Guan-Li \cite{GL1} generalized Alexandrov-Fenchel inequalities for star-shaped $k$-convex hypersurface in Euclidean space using inverse curvature flow (\ref{ICF}) in Euclidean space. Recently, Brendle-Hung-Wang \cite{BHW} generalized Alexandrov-Fenchel inequality for $k=1$ (which they call Minkowski inequality) in anti-de Sitter-Schwarzschild manifold by inverse mean curvature flow (\ref{IMCF}). The inequality is further used to prove a Penrose inequality in General Relativity in \cite{BW}. More recently, Li-Wei-Xiong \cite{LWX} and Ge-Wang-Wu\cite{GWW} generalized the hyperbolic Alexandrov-Fenchel inequality using inverse curvature flow (\ref{ICF}) in hyperbolic space.

Motivated by the results above, we consider inverse curvature flow in anti-de Sitter-Schwarzschild manifold. The anti-de Sitter-Schwarzschild manifold is a manifold $N=\mathbb{S}^n\times [s_0,\infty)$ equipped with the following Riemannian metric
\begin{align*}
\bar{g}=\frac{1}{1-ms^{1-n}+s^2}ds^2+s^2g_{\mathbb{S}^n}
\end{align*}
where $s_0$ is the unique positive solution of the equation $1-ms^{1-n}+s^2=0$. By a change of variable, we have
\begin{align*}
\bar{g}=dr^2+\phi^2(r)g_{\mathbb{S}^n}
\end{align*}
where $\phi$ satisfies $\phi^\prime=\sqrt{1-m\phi^{1-n}+\phi^2}$.

The anti-de Sitter-Schwarzschild manifold is thus a special case of warped product manifold. Moreover, the sectional curvature of $(N,\bar{g})$ approach $-1$ near infinity exponentially fast and the scalar curvature is of constant $-n(n+1)$. This feature will play an essential role in the proof of our theorem.

To state our theorem, we need the following definition of Garding's $\Gamma_k$ cone $\Gamma_k=\{(\kappa_i)\in \mathbb{R}^n|\sigma_j>0,0\leq j\leq k\}$, where $\sigma_j$ is the $j$-th elementrary symmetric function. We say a hypersurface is $k$-convex if the principal curvature $(\kappa_i)\in \Gamma_k$. 

We now state our main theorem:
\begin{theo}\label{thm1}
Let $\Sigma^n_0$ be a star-shaped, $k$-convex closed hypersurface in $N^{n+1}$, where $N^{n+1}$ is an anti-de Sitter-Schwarzschild manifold, consider the evolution equation
\begin{align}
\dot{X}=\frac{\nu}{F}
\end{align}
where $\nu$ is the ourward unit normal and $F=n\frac{C^{k-1}_n}{C^k_n}\frac{\sigma_k}{\sigma_{k-1}}$ which is evaluated at the principal curvature of $\Sigma_t$. Then the solution exists for all time $t$, and the second fundamental form satisfies
\begin{align*}
|h^i_j-\delta^i_j|\leq Ce^{-\frac{2}{n}t}
\end{align*}
where $C$ depends on the $\Sigma_0,n,k$.
\end{theo}

The organization of the paper is as follows: in section 2, we give some preliminaries about warped product space and anti-de Sitter-Schwarzschild manifold, we also prove the $C^0$ estimate. In section 3, we derive the evolution equations and give the $C^1$ estimate. In section 4 and 5, we estimate the bound for $F$ and the principal curvature respectively. In section 6, we prove that the second fundamental form converges to identity.

After submitting the paper, we have learned that Chen-Mao \cite{CM} independently proved the main theorem above.

\section{Preliminaries}
In this section, we give some basic properties of hypersurface in warped product space. Let $N^{n+1}$ be a warped product space, with the metric 
\begin{equation}
g^N:= ds^2=dr^2+\phi^2(r)\sigma_{ij}
\end{equation}
where $\sigma_{ij}$ is the standard metric of $\mathbb{S}^n$.

Define 
\begin{align*}
\Phi(r)=\int_0^r\phi(\rho)d\rho,\quad V=\phi(r)\frac{\partial}{\partial r}
\end{align*}
We state some well-known lemmas, see \cite{GL2} with some modification.
\begin{lemm}
The vector field $V$ satisfies $D_iV_j=\phi^\prime(r)g_{ij}^N$, where $D$ is the covariant derivative with respect to the metric $g^N$.
\end{lemm}
\begin{lemm}
Let $M^n\subset N^{n+1}$ be a closed hypersurface with induced metric $g$, then $\Phi|_M$ satisfies,
\begin{align*}
\nabla_i\nabla_j\Phi =\phi^\prime(r)g_{ij}-h_{ij}\left\langle V,\nu\right\rangle,
\end{align*}
where $\nabla$ is the covariant derivative with respect to $g$, $\nu$ is the outward unit normal and $h_{ij}$ is the second fundamental form of the hypersurface. 
\end{lemm}

We now state the Gauss equation Codazzi equation,
\begin{align}\label{Gauss}
R_{ijkl}=\bar{R}_{ijkl}+\left(h_{ik}h_{jl}-h_{il}h_{jk}\right)
\end{align}
\begin{align}\label{Codazzi}
\nabla_k h_{ij}-\nabla_j h_{ik}=\bar{R}_{\nu ijk}
\end{align}
and the interchanging formula
\begin{align}\label{interchanging formula}
\nabla_i\nabla_jh_{kl}=&\nabla_k\nabla_lh_{ij}-h^m_{l}(h_{im}h_{kj}-h_{ij}h_{mk})-h^m_{j}(h_{mi}h_{kl}-h_{il}h_{mk})\\\nonumber
&+h^m_{l}\bar{R}_{ikjm}+h^m_{j}\bar{R}_{iklm}+\nabla_k\bar{R}_{ijl\nu}+\nabla_i\bar{R}_{jkl\nu}
\end{align}

Define the support function $u=\left\langle V,\nu\right\rangle$, and we have
\begin{lemm}\label{support function}
\begin{align*}
\nabla_i u &=g^{kl}h_{ik}\nabla_l \Phi,\\
\nabla_i\nabla_j u&=g^{kl}\nabla_k h_{ij}\nabla_l\Phi+\phi^\prime h_{ij}-(h^2)_{ij}u+g^{kl}\nabla_l\Phi\bar{R}_{\nu jki},
\end{align*}
where $(h^2)_{ij}=g^{kl}h_{ik}h_{jl}$, $\bar{R}_{\nu jki}$ is the curvature of ambient space.
\end{lemm}
\begin{proof}
We only need to prove the equality at one point, thus we have $g_{ij}=\delta_{ij}$ and $\nabla_i u=D_i\left\langle V,\nu\right\rangle=\left\langle D\Phi,D_i\nu\right\rangle=h_{ik}D_k\Phi$.
\begin{align*}
\nabla_i\nabla_j u&=\nabla_i h_{jk}\nabla_k\Phi+h_{jk}\nabla_i\nabla_k\Phi\\
&=\nabla_i h_{jk}\nabla_k\Phi+h_{jk}(\phi^\prime g_{ik}-h_{ik}u)\\
&=\left(\nabla_k h_{ij}+\bar{R}_{\nu jki}\right)\nabla_k\Phi+\phi^\prime h_{ij}-(h^2)_{ij}u,
\end{align*}
where Codazzi equation (\ref{Codazzi}) is used in the last equality, thus by the tensorial property, we have the lemma.
\end{proof}

As to the curvature, we have the following curvature estimates, for proof, we refer readers to \cite{BHW}.
\begin{lemm}\label{ambient curvature}
The sectional curvature satisfies 
\begin{align*}
&\bar{R}(\partial_i,\partial_j,\partial_k,\partial_l)=\phi^2\left(1-{\phi^\prime}^2\right)(\sigma_{ik}\sigma_{jl}-\sigma_{il}\sigma_{jk})\\
&\bar{R}(\partial_i,\partial_r,\partial_j,\partial_r)=-\phi\phi^{\prime\prime}\sigma_{ij}
\end{align*}
where $\partial_i$ is the standard frame on $\mathbb{S}^n$ and $\sigma_{ij}$ is the standard metric of $\mathbb{S}^n$.
\end{lemm}
Now, back to our case that $N$ is an anti-de Sitter-Schwarzschild manifold, 
\begin{lemm}\label{ADS curvature}
Let $N$ be an anti-de Sitter-Schwarzschild manifold, we have
\begin{align}
\phi(r)=\sinh(r)+\frac{m}{2(n+1)}\sinh^{-n}(r)+O(\sinh^{-n-2}(r))
\end{align}
and
\begin{align*}
\bar{R}_{\alpha\beta\gamma\mu}=-\delta_{\alpha\gamma}\delta_{\beta\mu}+\delta_{\alpha\mu}\delta_{\beta\gamma}+O(e^{-(n+1)r})
\end{align*}
\begin{align*}
\bar{\nabla}_\rho \bar{R}_{\alpha\beta\gamma\mu}=O(e^{-(n+1)r})
\end{align*}
where $\{e_\alpha\}$ is an orthonormal frame in $N$.
\end{lemm}

We also need the following two lemmas regarding to $\sigma_k$. These two lemmas are well known, for completeness, we add the proof here.
\begin{lemm}\label{sigma_k 1}
let $F=n\frac{C^{k-1}_n}{C^k_n}\frac{ \sigma_k}{\sigma_{k-1}}$, thus $F$ is of homogeneous degree $1$, and $F(I)=n$, then we have
\begin{align*}
\sum_i F^{ii}\lambda_i^2\geq \frac{F^2}{n}
\end{align*}
\end{lemm}

\begin{proof}
We first consider the term $\sigma_l^{ii}\lambda_i^2$, we have
\begin{align}\label{sigma_k}
\sigma_l^{ii}\lambda_i^2=\sigma_1\sigma_l-(l+1)\sigma_{l+1}
\end{align}

Let $G=\frac{\sigma_k}{\sigma_{k-1}}$,  by (\ref{sigma_k}) and Newton-Mclaraun ineqaulity, we have
\begin{align*}
\sum_i G^{ii}\lambda_i^2&=\sum_i\left(\frac{\sigma_k^{ii}}{\sigma_{k-1}}-\frac{\sigma_k\sigma_{k-1}^{ii}}{\sigma_{k-1}^2}\right)\lambda_i^2\\
&=\frac{\sigma_1\sigma_k-(k+1)\sigma_{k+1}}{\sigma_{k-1}}-\frac{\sigma_k\left(\sigma_1\sigma_{k-1}-k\sigma_{k}\right)}{\sigma_{k-1}^2}\\
&=\frac{k\sigma_k^2-(k+1)\sigma_{k-1}\sigma_{k+1}}{\sigma_{k-1}^2}\\
&\geq\frac{k\sigma_k^2}{(n-k+1)\sigma_{k-1}^2}\\
&=\frac{C^{k-1}_n}{C^k_n}\left(\frac{\sigma_k}{\sigma_{k-1}}\right)^2
\end{align*}
thus
\begin{align*}
\sum_i F^{ii}\lambda_i^2\geq n\left(\frac{C^{k-1}_n}{C^k_n}\right)^2\left(\frac{\sigma_k}{\sigma_{k-1}}\right)^2=\frac{F^2}{n}
\end{align*}

\end{proof}

\begin{lemm}\label{sigma_k 2}
Let $F=n\frac{C^{k-1}_n}{C^k_n}\frac{\sigma_k}{\sigma_{k-1}}$ and $(\lambda_i)\in \Gamma_k$, then 
\begin{align*}
n\leq \sum_i F^{ii}\leq nk
\end{align*}
\end{lemm}
\begin{proof}
Let $G=\frac{\sigma_k}{\sigma_{k-1}}$, we have
\begin{align*}
\sum_i G^{ii}&=\sum_i\left(\frac{\sigma_k^{ii}}{\sigma_{k-1}}-\frac{\sigma_k\sigma_{k-1}^{ii}}{\sigma_{k-1}^2}\right)\\
&=(n-k+1)-(n-k+2)\frac{\sigma_k\sigma_{k-2}}{\sigma_{k-1}^2}\\
&\geq \frac{n-k+1}{k}
\end{align*}
by Newton-Mclaraun inequality.

For the second inequality,
\begin{align*}
\sum_i G^{ii}&=\sum_i\left(\frac{\sigma_k^{ii}}{\sigma_{k-1}}-\frac{\sigma_k\sigma_{k-1}^{ii}}{\sigma_{k-1}^2}\right)\\
&=(n-k+1)-(n-k+2)\frac{\sigma_k\sigma_{k-2}}{\sigma_{k-1}^2}\\
&\leq n-k+1
\end{align*}
as $(\lambda_i)\in \Gamma_k$. The lemma then follows.
\end{proof}

Since the initial hypersurface is star-shaped, we can consider it as a graph on $\mathbb{S}^n$, i.e. $X=(x,r)$ where $x$ is the coordinate on $\mathbb{S}^n$, $r$ is the radius, by taking derivatives, we have
\begin{align}\label{e_i}
&X_i=\partial_i+r_i\partial_r\\\nonumber
&g_{ij}=r_ir_j+\phi^2\sigma_{ij}
\end{align}
and
\begin{align}\label{v}
\nu=\frac{1}{v}\left (-\frac{r^i}{\phi^2}\partial_i+\partial_r\right )
\end{align}
where $\nu$ is the unit normal vector, $v=(1+\frac{|\nabla r|^2}{\phi^2})^{\frac{1}{2}}$, note that all the derivatives are on $\mathbb{S}^n$.

Thus
\begin{align*}
\frac{dr}{dt}=\frac{1}{Fv}, \dot{x}^i=-\frac{r^i}{\phi^2Fv}
\end{align*}
we have
\begin{align}\label{4.1}
\frac{\partial r}{\partial t}=\frac{dr}{dt}-r_j \dot{x}^j=\frac{v}{F}
\end{align}

By a direct computation, c.f. (2.6) in \cite{D} we have
\begin{align}\label{4.2}
h_{ij}=\frac{1}{v}(-r_{ij}+\phi\phi^\prime\sigma_{ij}+\frac{2\phi^\prime r_ir_j}{\phi})
\end{align}

Now we consider a function 
\begin{align}
\varphi=\int_{r_0}^r\frac{1}{\phi}
\end{align}
thus
\begin{align}\label{4.3}
\varphi_i=\frac{r_i}{\phi}, \varphi_{ij}=\frac{r_{ij}}{\phi}-\frac{\phi^\prime r_ir_j}{\phi^2}.
\end{align}
If we write everything in terms of $\varphi$, we have
\begin{align}\label{derivative of varphi}
\frac{\partial \varphi}{\partial t}=\frac{v}{\phi F}
\end{align}
and
\begin{align}
v=(1+|D\varphi|^2)^{\frac{1}{2}},g_{ij}=\phi^2(\varphi_i\varphi_j+\sigma_{ij}), g^{ij}=\phi^{-2}\left(\sigma^{ij}-\frac{\varphi^i\varphi^j}{v^2}\right).
\end{align}

Moreover, 
\begin{align}\label{4.6}
h_{ij}&=\frac{\phi}{v}\left( \phi^\prime(\sigma_{ij}+\varphi_i\varphi_j)-\varphi_{ij}\right),\\\nonumber
h^i_j&=g^{ik}h_{kj}=\frac{\phi^\prime}{\phi v}\delta^i_j-\frac{1}{\phi v}\tilde{\sigma}^{ik}\varphi_{kj}
\end{align}
where $\tilde{\sigma}^{ij}=\sigma^{ij}-\frac{\varphi^i\varphi^j}{v^2}$.

We now give the $C^0$ estimate.
\begin{lemm}\label{C^0 estimate}
Let $\bar{r}(t)=\sup_{\mathbb{S}^n}r(\cdot,t)$ and $\underline{r}(t)=\inf_{\mathbb{S}^n}r(\cdot,t )$, then we have
\begin{align}
&\phi(\bar{r}(t))\leq e^{t/n}\phi(\bar{r}(0))\\\nonumber
&\phi(\underline{r}(t))\geq e^{t/n}\phi(\underline{r}(0))
\end{align}
\end{lemm}
\begin{proof}
Recall that $\frac{\partial r}{\partial t}=\frac{v}{F}$, where $F$ is a normalized operator on $(h^i_j)$. At the point where the function $r(\cdot, t)$ attains its maximum, we have $\nabla r=0, (r_{ij})\leq 0$, from (\ref{4.3}), we deduce that $\nabla \varphi=0, (\varphi_{ij})\leq 0$ at the maximum point. From (\ref{4.6}), we have $(h^i_j)\geq \left(\frac{\phi^\prime}{\phi }\delta^i_j\right)$, where we may assume $(g_{ij})$ and $(h_{ij})$ is diagonalized if necessary. Since $F$ is homogeneous of degree $1$, and $F(1,\cdots,1)=n$, we have
\begin{align*}
v^2=1+|\nabla\varphi|^2=1, F(h^i_j)\geq \frac{\phi^\prime}{\phi }F(\delta^i_j)=\frac{n\phi^\prime}{\phi },
\end{align*}
thus
\begin{align*}
\frac{d}{dt}\bar{r}(t)\leq \frac{\phi(\bar{r}(t))}{n\phi^\prime (\bar{r}(t))}
\end{align*}
i.e.
\begin{align*}
\frac{d}{dt}\log\phi(\bar{r}(t))\leq \frac{1}{n}
\end{align*}
which yields to the first inequality. Similarly, we can prove the second inequality, thus we have the lemma.
\end{proof}

\section{Evolution equations and $C^1$ estimate}
Before we go on with the estimate, let's derive some evolution equations first. 
\begin{align}
\dot{g}_{ij}=\frac{2h_{ij}}{F},\quad \dot{\nu}=\frac{g^{ij}F_ie_j}{F^2}
\end{align}

\begin{align}\label{6.5}
\dot{h}^i_j=-\frac{1}{F}h^{i}_kh^k_{j}-\nabla^i\nabla_j\left(\frac{1}{F}\right)-\frac{1}{F}\bar{R}^i_{\nu j\nu}
\end{align}

Together with the interchanging formula (\ref{interchanging formula}), we have
\begin{align}\label{6.6}
\dot{h}^i_j&=-\frac{1}{F}h^i_kh^k_j+ \frac{F^{pq,rs}{h_{pq}}^ih_{rsj}}{F^2}  -\frac{2F^{pq}{h_{pq}}^iF^{rs}h_{rsj}}{F^3}-\frac{1}{F}\bar{R}^i_{\nu j\nu}\\\nonumber
&+\frac{g^{ki}F^{pq}}{F^2} \big( h_{kj,pq}-h^m_{q}(h_{km}h_{pj}-h_{kj}h_{mp})-h^m_{j}(h_{mk}h_{pq}-h_{kq}h_{mp})\\\nonumber
&+h^m_{q}\bar{R}_{kpjm}+h^m_{j}\bar{R}_{kpqm}+\nabla_p\bar{R}_{kjq\nu}+\nabla_k\bar{R}_{jpq\nu} \big)
\end{align}
where $F^{ij}=\frac{\partial F}{\partial h_{pq}}$ and $F^{pq,rs}=\frac{\partial^2 F}{\partial h_{pq}\partial h_{rs}}$.

For later purpose, we consider the function $u=\left \langle \phi \partial_r, \nu\right \rangle=\frac{\phi}{v}$, which can be seen as the support function. We derive the following equation.

\begin{align}\label{support derivative}
\dot{u}&=\frac{\phi^\prime}{F}+\frac{\phi  g^{ij}F_ir_j}{F^2}
\end{align}

Now, we need to consider the curvature term. By Lemma \ref{ambient curvature}, (\ref{e_i}) and (\ref{v}), we have

\begin{align}\label{6.9}
\bar{R}_{k \nu j\nu }&=\left(\frac{1}{v^2}\delta_{kj}+\frac{2r_kr_j}{\phi^2v^2}+\frac{r_kr_j|\nabla r|^2}{\phi^4v^2}\right)(-\phi\phi^{\prime\prime})+\frac{(|\nabla r|^2\delta_{kj}-r_kr_j)}{\phi^2v^2}(1-{\phi^\prime}^2)
\end{align}

\begin{align*}
\bar{R}_{\nu jnk}&=\frac{r_n \delta_{jk}}{v}\left(-\phi\phi^{\prime\prime}-(1-(\phi^\prime)^2)\right)+\frac{r_k \delta_{jn}}{v}\left(\phi\phi^{\prime\prime}+(1-{\phi^\prime}^2)\right)
\end{align*}

Note that $g^{mn}=\phi^{-2}\left(\sigma_{mn}-\frac{r^mr^n}{v^2\phi^2}\right)$, thus
\begin{align}\label{6.10}
g^{mn}\nabla_m\Phi\bar{R}_{\nu jnk}&=\left(\frac{|\nabla r|^2\delta_{jk}-r_jr_k}{\phi v^3}\right)\left(-\phi\phi^{\prime\prime}-(1-{\phi^\prime}^2)\right)
\end{align}

\begin{lemm}\label{time derivative}
Along the flow, $|\dot{\varphi}|\leq C$, where $C$ depends on $\Sigma_0, n, k$.
\end{lemm}
\begin{proof}
By (\ref{derivative of varphi}) and (\ref{4.6}), we have
\begin{align*}
\frac{\partial \varphi}{\partial t}=\frac{v^2}{F(\phi^\prime \delta_{ij}-\tilde{\sigma}^{ik}\varphi_{kj})}=\frac{1}{G}
\end{align*}
Let $G^{ij}=\frac{\partial G}{\partial \varphi_{ij}}$, $G^k=\frac{\partial G}{\partial \varphi_k}$, then
\begin{align*}
G^{ij}=-\frac{1}{v^2}F^{i}_l\tilde{\sigma}^{lj}
\end{align*}
thus
\begin{align*}
\frac{\partial \dot{\varphi}}{\partial t}=-\frac{\dot{G}}{G^2}=\frac{1}{v^2G^2}\left(F^{i}_l\tilde{\sigma}^{lj}\dot{\varphi}_{ij}-v^2 G^k\dot{\varphi}_k-F^{i}_i\phi\phi^{\prime\prime} \dot{\varphi}\right)
\end{align*}
By maximum principle, we conclude that $|\dot{\varphi}|$ is bounded above.
\end{proof}

\begin{lemm}\label{C^1 estimate}
Along the flow, $|\nabla \varphi|\leq C$, where $C$ depends on $\Sigma_0,n,k$. In addition, if $F$ is bounded above, we have $|\nabla \varphi|\leq Ce^{-\alpha t}$, where $\alpha$ depends on $\sup F$ and $n$.
\end{lemm}

\begin{proof}
By (\ref{derivative of varphi}) and (\ref{4.6}), we have
\begin{align*}
\frac{\partial \varphi}{\partial t}=\frac{v^2}{F(\phi^\prime \delta_{ij}-\tilde{\sigma}^{ik}\varphi_{kj})}=\frac{1}{G}
\end{align*}
Let $G^{ij}=\frac{\partial G}{\partial \varphi_{ij}}$, $G^k=\frac{\partial G}{\partial \varphi_k}$, then
\begin{align*}
G^{ij}=-\frac{1}{v^2}F^{i}_l\tilde{\sigma}^{lj}
\end{align*}
Let $\omega=\frac{1}{2}|\nabla \varphi|^2$, we have
\begin{align*}
\frac{\partial \omega}{\partial t}=-\frac{\varphi^k}{G^2}\nabla_kG=\frac{1}{v^2G^2}\left(F^{i}_l\tilde{\sigma}^{lj}\varphi^k\varphi_{ijk}-v^2 G^k\omega_k-2F^{i}_i\phi\phi^{\prime\prime} \omega\right)
\end{align*}

We want to write the term $\tilde{\sigma}^{lj}\varphi_{ijk}$ in terms of second derivative of $\omega$. Note that
\begin{align*}
\omega_{ij}&=\varphi_{kij}\varphi^k+\varphi_{ki}\varphi^{k}_j\\
&=\varphi_{ijk}\varphi^k+(\sigma_{ij}\sigma_{kp}-\sigma_{ik}\sigma_{jp})\varphi^p\varphi^k+\varphi_{ki}\varphi^{k}_j\\
&=\varphi_{ijk}\varphi^k+\sigma_{ij}|\nabla\varphi |^2-\varphi_i\varphi_j+\varphi_{ki}\varphi^{k}_j
\end{align*}
and
\begin{align*}
\tilde{\sigma}^{lj}\left(\sigma_{ij}|\nabla\varphi |^2-\varphi_i\varphi_j\right)=\delta_{i}^l|\nabla \varphi|^2-\varphi_i\varphi^l
\end{align*}
Thus we have
\begin{align*}
\frac{\partial w}{\partial t}=\frac{1}{v^2G^2}\left(F^{i}_l\tilde{\sigma}^{lj}\omega_{ij}-F^{i}_i|\nabla \varphi|^2+F^{i}_l\varphi_i\varphi^l-v^2 G^k\omega_k-2F^{i}_i\phi\phi^{\prime\prime} \omega\right)-\frac{1}{v^2G^2}F^{i}_l\tilde{\sigma}^{lj}\varphi_{ki}\varphi^{k}_j
\end{align*}

Note that $-F^{i}_i|\nabla \varphi|^2+F^{i}_l\varphi_i\varphi^l\leq 0$ and $-F^{i}_l\tilde{\sigma}^{lj}\varphi_{ki}\varphi^{k}_j\leq 0$, thus by the maximum principle, we have
\begin{align*}
\omega(\cdot,t)\leq \sup\omega_0
\end{align*}

More pricisely, if $F\leq C$, consider the test function $\tilde{\omega}=\omega e^{\lambda t}$, thus at the maximum point of $\tilde{\omega}$, we have
\begin{align*}
0&\leq \frac{\partial \omega}{\partial t}e^{\lambda t}+\lambda \omega e^{\lambda t}\leq \omega e^{\lambda t}\left(\frac{-2F^{i}_i\phi\phi^{\prime\prime}}{v^2G^2}+\lambda\right)\\
&=  \omega e^{\lambda t}\left(\frac{-2F^{i}_i(h^i_j)\phi^{\prime\prime} }{\phi F^2(h^i_j)}+\lambda\right)\\
&\leq \omega e^{\lambda t}\left(\frac{-2n\phi^{\prime\prime} }{\phi F^2(h^i_j)}+\lambda\right)\leq 0
\end{align*}
if $0<\lambda\leq \frac{2n}{sup^2 F}\leq \frac{2n\phi^{\prime\prime} }{\phi sup^2F}$, we have used Lemma \ref{sigma_k 2} in last line. By maximum principle,
\begin{align*}
|\nabla \varphi|\leq Ce^{-\alpha t}
\end{align*}
where $0<\alpha\leq \frac{n}{sup^2F}$.
\end{proof}

\section{bound for $F$}
\begin{lemm}\label{lemma 8}
Along the flow, $F\leq C$, where $C$ depends on $\Sigma_0,n,k$.
\end{lemm}

\begin{proof}
By (\ref{6.5}), we have
\begin{align*}
\dot{F}&=F^{j}_i\left(-\frac{1}{F}h^{i}_kh^k_{j}-\nabla^i\nabla_j\left(\frac{1}{F}\right)-\frac{1}{F}\bar{R}^i_{\nu j\nu}\right)\\
&=F^{j}_i\left(-\frac{1}{F}h^{i}_kh^k_{j}+\frac{\nabla^i\nabla_jF}{F^2}-2\frac{\nabla^i F\nabla_jF}{F^3}-\frac{1}{F}\bar{R}^i_{\nu j\nu}\right)
\end{align*}
By Lemma \ref{sigma_k 1}, we have
\begin{align*}
\dot{F}\leq -\frac{F}{n}+F^{j}_i\left(\frac{\nabla^i\nabla_jF}{F^2}-2\frac{\nabla^i F\nabla_jF}{F^3}-\frac{1}{F}\bar{R}^i_{\nu j\nu}\right)
\end{align*}

By Lemma \ref{ambient curvature}, we know that $\bar{R}^i_{\nu j\nu}$ is uniformly bounded, together with Lemma \ref{sigma_k 2}, we have
\begin{align*}
-F^{j}_i\bar{R}^i_{\nu j\nu}\leq C\sum_i F^{ii}\leq C
\end{align*}
thus we get
\begin{align*}
\dot{F}^2_{max}\leq -\frac{2}{n}F^2_{max}+C
\end{align*}
which gives
\begin{align*}
F^2_{max}\leq C
\end{align*}

\end{proof}

\begin{lemm}\label{lemma 9}
Along the flow, $F\geq c$, where $c$ depends on $\Sigma_0,n,k$.
\end{lemm}

\begin{proof}
Consider the function $-\log F-\log \tilde{u}$, where $\tilde{u}=ue^{-t/n}$, by Lemma \ref{C^0 estimate}, $\tilde{u}$ is uniformly bounded. At the maximum point, we have
\begin{align*}
&-\frac{F_i}{F}-\frac{u_i}{u}=0, -\frac{F_{ij}}{F}+\frac{F_iF_j}{F^2}-\frac{u_{ij}}{u}+\frac{u_iu_j}{u^2}\leq 0\\
&-\frac{F^j_i}{F}\dot{h}^i_j-\frac{\dot{u}}{u}+\frac{1}{n}\geq 0
\end{align*}
by (\ref{6.5}), (\ref{support derivative}) and the critical equation, we have
\begin{align*}
0&\leq -\frac{F^j_i}{F} \left(-\frac{1}{F}h^{i}_kh^k_{j}-\nabla^i\nabla_j\left(\frac{1}{F}\right)-\frac{1}{F}\bar{R}^i_{\nu j\nu}\right)-\frac{\phi^\prime}{Fu}-\frac{\phi  g^{ij}F_ir_j}{F^2u}+\frac{1}{n}\\
&= \frac{F^j_i}{F^2}\left(h^i_kh^k_j+\bar{R}^i_{\nu j\nu}\right)+\frac{g^{ki}F^j_i}{F^2}\left(-\frac{F_{kj}}{F}+2\frac{F_kF_j}{F^2}\right)-\frac{\phi^\prime}{Fu}-\frac{\phi  g^{ij}F_ir_j}{F^2u}+\frac{1}{n}\\
&\leq \frac{F^j_i}{F^2}\left(h^i_kh^k_j+\bar{R}^i_{\nu j\nu}\right)+\frac{g^{ki} F^j_i}{F^2}\frac{u_{kj}}{u}-\frac{\phi^\prime}{Fu}-\frac{\phi  g^{ij}F_ir_j}{F^2u}+\frac{1}{n}
\end{align*}
by lemma \ref{support function}, we have
\begin{align*}
0&\leq \frac{F^j_i}{F^2}\left(h^i_kh^k_j+\bar{R}^i_{\nu j\nu}\right)+\frac{g^{ki} F^j_i}{F^2u}\left(  g^{mn} h_{kjm}\phi r_n+\phi^\prime h_{kj}-(h^2)_{kj}u+g^{mn}\nabla_m\Phi\bar{R}_{\nu jnk}\right)\\
&-\frac{\phi^\prime}{Fu}-\frac{\phi  g^{ij}F_ir_j}{F^2u}+\frac{1}{n}\\
&= \frac{F^j_i\bar{R}^i_{\nu j\nu }}{F^2}+\frac{g^{ki}F^j_i}{F^2u}  g^{mn}\nabla_m\Phi\bar{R}_{\nu jnk}+\frac{1}{n}
\end{align*}
by (\ref{6.9}) and (\ref{6.10}), we have
\begin{align*}
0&\leq \frac{g^{ki}F^j_i}{F^2}\bigg(\left(\frac{1}{v^2}\delta_{kj}+\frac{2r_kr_j}{v^2\phi^2}+\frac{r_kr_j|\nabla r|^2}{v^2\phi^4}\right)(-\phi\phi^{\prime\prime})+\frac{(|\nabla r|^2\delta_{kj}-r_kr_j)}{v^2\phi^2}(1-{\phi^\prime}^2)\\
&+\left(\frac{|\nabla r|^2\delta_{jk}-r_jr_k}{v^2\phi^2}\right)\left(-\phi\phi^{\prime\prime}-(1-{\phi^\prime}^2)\right)\bigg)+\frac{1}{n}\\
&=\frac{g^{ki}F^j_i}{F^2}\left(\delta_{kj}+\frac{r_kr_j}{\phi^2}\right)(-\phi\phi^{\prime\prime})+\frac{1}{n}\\
&\leq -\frac{g^{ij}F^{j}_i}{F^2}\phi\phi^{\prime\prime}+\frac{1}{n}\leq -\frac{C}{F}+\frac{1}{n}
\end{align*}
we have used the Lemma \ref{sigma_k 2} in last line. Now we conclude that $F$ is bounded below.
\end{proof}

\begin{rema}
For the lower bound, we only need the first inequality of Lemma \ref{sigma_k 2}, which is satisfied by a class of concave functions with homogeneous degree one, for example $F=\sigma_k^{1/k}$, etc.
\end{rema}

\section{bound for principal curvature}
\begin{lemm}\label{princial curvature}
Along the flow, $|\kappa_i|\leq C$ if $F$ is a hessian quotient function, where $\kappa_i$ is the principal curvature of $\Sigma_t$, $C$ depends on $\Sigma_0,n,k$.
\end{lemm}

\begin{proof}
Define $\tilde{u}=ue^{-t/n}$, consider the test function $\log(\eta)-\log(\tilde{u})$, where
\begin{align*}
\eta=\sup\{h_{ij}\xi^i\xi^j:g_{ij}\xi^i\xi^j=1\}
\end{align*}
WLOG, we suppose that at the maximum point $\eta=h^1_1$, and we have
\begin{align}\label{9.1}
\frac{\dot{h_1^1}}{h^1_1}-\frac{\dot{u}}{u}+\frac{1}{n}\geq 0 
\end{align}
and
\begin{align}\label{9.2}
\frac{h^1_{1i}}{h^1_1}-\frac{u_i}{u}=0,\quad\frac{h^1_{1ij}}{h^1_1}\leq \frac{u_{ij}}{u}
\end{align}
by (\ref{6.6}), (\ref{support derivative}) and the critical equation, we have
\begin{align}\label{9.3}
0&\leq \frac{1}{h^1_1}\bigg(-\frac{1}{F}h^1_kh^k_1+ \frac{F^{pq,rs}{h_{pq}}^1h_{rs1}}{F^2}  -\frac{2F^{pq}{h_{pq}}^1F^{rs}h_{rs1}}{F^3}-\frac{1}{F}\bar{R}^1_{\nu 1\nu}\\\nonumber
&+\frac{g^{k1}F^{pq}}{F^2} \big( h_{k1,pq}-h^m_{q}(h_{km}h_{p1}-h_{k1}h_{mp})-h^m_{1}(h_{mk}h_{pq}-h_{kq}h_{mp})\\\nonumber
&+h^m_{q}\bar{R}_{kp1m}+h^m_{1}\bar{R}_{kpqm}+\nabla_p\bar{R}_{k1q\nu}+\nabla_k\bar{R}_{1pq\nu} \big)\bigg)\\\nonumber
&-\frac{\phi^\prime}{Fu}-\frac{\phi  g^{ij}F_ir_j}{F^2u}+\frac{1}{n}
\end{align}
consider the term $\frac{F^{pq}}{F^2} \frac{h^1_{1,pq}}{h^1_1} $, by (\ref{9.2}) and lemma \ref{support function},  we have
\begin{align}\label{9.4}
\frac{F^{pq}}{F^2} \frac{h^1_{1,pq}}{h^1_1} \leq \frac{F^{pq}}{F^2}\frac{u_{pq}}{u}=\frac{ F^{pq}}{F^2u}\left(g^{kl}h_{pqk}\Phi_l+\phi^\prime h_{pq}-(h^2)_{pq}u+g^{kl}\nabla_l\Phi\bar{R}_{\nu pkq}\right)
\end{align}
insert (\ref{9.4}) into (\ref{9.3}), together with the concavity of $F$, yields
\begin{align}
0&\leq \frac{1}{h^1_1}\bigg(-\frac{1}{F}h^1_kh^k_1-\frac{1}{F}\bar{R}^1_{\nu 1\nu}+\frac{g^{k1}F^{pq}}{F^2} \big(-h^m_{1}h_{mk}h_{pq}+h^m_{q}\bar{R}_{kp1m}+h^m_{1}\bar{R}_{kpqm}+\nabla_p\bar{R}_{k1q\nu}+\nabla_k\bar{R}_{1pq\nu} \big)\bigg)\\\nonumber
&+\frac{ g^{kl}F^{pq}}{F^2u}\nabla_l\Phi\bar{R}_{\nu pkq}+\frac{1}{n}
\end{align}

Using the fact $1-{\phi^\prime}^2+\phi\phi^{\prime\prime}\geq 0$, together with (\ref{6.10})
\begin{align}\label{9.5}
g^{kl}\nabla_l\Phi\bar{R}_{\nu pkq}=\left(\frac{|\nabla r|^2\delta_{pq}-r_pr_q}{v^3\phi}\right)\left(-\phi\phi^{\prime\prime}-(1-{\phi^\prime}^2)\right)\leq 0
\end{align}
thus we have
\begin{align}\label{9.6}
0&\leq \frac{1}{h^1_1}\bigg(-\frac{2}{F}h^1_kh^k_1-\frac{1}{F}\bar{R}^1_{\nu 1\nu}+\frac{g^{k1}F^{pq}}{F^2} \big(h^m_{q}\bar{R}_{kp1m}+h^m_{1}\bar{R}_{kpqm}+\nabla_p\bar{R}_{k1q\nu}+\nabla_k\bar{R}_{1pq\nu} \big)\bigg)+\frac{1}{n}
\end{align}

By Lemma \ref{ADS curvature}, all terms involving curvature terms of the ambient space are uniformly bounded, i.e.
\begin{align*}
h^m_{q}\bar{R}_{kp1m}+h^m_{1}\bar{R}_{kpqm}+\nabla_p\bar{R}_{k1q\nu}+\nabla_k\bar{R}_{1pq\nu} \leq Ch^1_1+C
\end{align*}

By Lemma \ref{sigma_k 2} and the lower bound of $F$  Lemma \ref{lemma 9}, 
\begin{align*}
\frac{g^{k1}F^{pq}}{F^2} \big(h^m_{q}\bar{R}_{kp1m}+h^m_{1}\bar{R}_{kpqm}+\nabla_p\bar{R}_{k1q\nu}+\nabla_k\bar{R}_{1pq\nu} \big)\leq Ch^1_1+C
\end{align*}
Plug into (\ref{9.6}), together with the upper bound of $F$ Lemma \ref{lemma 8} yields
\begin{align*}
0\leq -Ch^1_1+C
\end{align*}
i.e. $h^1_1\leq C$, thus we have the lemma.

\end{proof}

\begin{coro}
The solution of the inverse curvature flow exists for all time.
\end{coro}
\begin{proof}
We have established up to $C^2$ apriori estimate, by Lemma \ref{princial curvature}, $F$ is uniformly elliptic, by Evans-Krylov theorem, we have $C^{2,\alpha}$ estimate, together with Schauder estimate, we have all the high order estimates, the corollary now follows.
\end{proof}

\section{Asmptotic behavior of second fundamental form}
In this section, we consider the asmptotic behaviour of second fundamental form, the test function was first considered by Scheuer in \cite{S1}.
\begin{lemm}\label{upbound of F}
\begin{align*}
\limsup_{t\rightarrow \infty}\sup_i\kappa_i\leq 1,
\end{align*}
where $\kappa_i$ is the principal curvature of $M$.
\end{lemm}
\begin{proof}
Let's consider the test function $w=\left(\log \eta-\log\tilde{u}+r-\log 2 \right)t $, where 
\begin{align*}
\eta=\sup\{h_{ij}\xi^i\xi^j:g_{ij}\xi^i\xi^j=1\}
\end{align*}

Noting that 
\begin{align*}
\left(-\log\tilde{u}+r-\log 2 \right)t=\left(\log v-\log\phi+r-\log 2\right)t
\end{align*}
by Lemma \ref{C^1 estimate}, $t\log v\leq C$. By Lemma \ref{ADS curvature}, we have
\begin{align*}
\phi\geq \frac{e^r}{2}-Ce^{-r}
\end{align*}
thus
\begin{align*}
\left(-\log\phi+r-\log 2\right)t\leq t\log {\frac{e^r}{e^r-Ce^{-r}}}\leq t\log\left( 1+Ce^{-2r}\right)\leq C
\end{align*}
i.e. 
\begin{align}\label{up bound of}
\left(-\log\tilde{u}+r-\log 2 \right)t\leq C
\end{align}
Similarly,
\begin{align}\label{lower bound of}
\left(-\log\tilde{u}+r-\log 2 \right)t\geq -C
\end{align}

WLOG, we suppose that at the maximum point of $w$, say $(x_0,t_0)$, $\eta=h^1_1$, and we have
\begin{align}
0\leq \left(\frac{\dot{h^1_1}}{h^1_1}-\frac{\dot{u}}{u}+\dot{r}\right)t+\left(\log h^1_1-\log\tilde{u}+r-\log 2 \right)
\end{align}
and
\begin{align}
&\frac{h^1_{1i}}{h^1_1}-\frac{u_i}{u}+r_i=0\\
&\frac{h^1_{1ij}}{h^1_1}-\frac{h^1_{1i}h^1_{1j}}{(h^1_1)^2}-\frac{u_{ij}}{u}+\frac{u_iu_j}{u^2}+r_{ij}\leq 0\nonumber
\end{align}
by (\ref{4.1}), (\ref{6.6}), (\ref{support derivative}) and the critical equation, we have
\begin{align*}
0&\leq \frac{t_0}{h^1_1}\bigg(-\frac{1}{F}h^1_kh^k_1+ \frac{F^{pq,rs}{h_{pq}}^1h_{rs1}}{F^2}  -\frac{2F^{pq}{h_{pq}}^1F^{rs}h_{rs1}}{F^3}-\frac{1}{F}\bar{R}^1_{\nu 1\nu}\\\nonumber
&+\frac{g^{k1}F^{pq}}{F^2} \big( h_{k1,pq}-h^m_{q}(h_{km}h_{p1}-h_{k1}h_{mp})-h^m_{1}(h_{mk}h_{pq}-h_{kq}h_{mp})\\\nonumber
&+h^m_{q}\bar{R}_{kp1m}+h^m_{1}\bar{R}_{kpqm}+\nabla_p\bar{R}_{k1q\nu}+\nabla_k\bar{R}_{1pq\nu} \big)\bigg)\\\nonumber
&-\frac{t_0}{u}\left(\frac{\phi^\prime}{F}+\frac{\phi  g^{ij}F_ir_j}{F^2}\right)+\frac{vt_0}{F}+\left(\log h^1_1-\log\tilde{u}+\tilde{r}-\log 2 \right)
\end{align*}
i.e.
\begin{align}\label{10.2}
0&\leq \frac{t_0}{h^1_1}\bigg(-\frac{2}{F}h^1_kh^k_1-\frac{1}{F}\bar{R}^1_{\nu 1\nu}+\frac{g^{k1}F^{pq}}{F^2} \big( h_{k1,pq}+h^m_{q}h_{k1}h_{mp}\\\nonumber
&+h^m_{q}\bar{R}_{kp1m}+h^m_{1}\bar{R}_{kpqm}+\nabla_p\bar{R}_{k1q\nu}+\nabla_k\bar{R}_{1pq\nu} \big)\bigg)\\\nonumber
&-\frac{t_0}{u}\left(\frac{\phi^\prime}{F}+\frac{\phi  g^{ij}F_ir_j}{F^2}\right)+\frac{vt_0}{F}+C
\end{align}

consider the term $\frac{F^{pq}}{F^2} \frac{h^1_{1,pq}}{h^1_1} $,  , by (\ref{4.2}), Lemma \ref{support function} and the critical equation we have
\begin{align}
\frac{F^{pq}}{F^2} \frac{h^1_{1,pq}}{h^1_1} &\leq \frac{F^{pq}}{F^2} \left(\frac{u_{pq}}{u}+\frac{h^1_{1p}h^1_{1q}}{(h^1_1)^2}-\frac{u_pu_q}{u^2}-r_{pq}\right)\\\nonumber
&=\frac{F^{pq}}{F^2u}\left(g^{kl}h_{pqk}\Phi_l+\phi^\prime h_{pq}-(h^2)_{pq}u+g^{kl}\nabla_l\Phi\bar{R}_{\nu pkq}\right)\\\nonumber
&+\frac{ F^{pq}}{F^2}\left(h_{pq}v-\phi\phi^\prime\delta_{pq}-\frac{2\phi^\prime r_pr_q}{\phi}\right)+\frac{F^{pq}}{F^2}\left(\frac{h^1_{1p}h^1_{1q}}{(h^1_1)^2}-\frac{u_pu_q}{u^2} \right)
\end{align}
plug into (\ref{10.2}), we have
\begin{align}\label{fffformula}
0&\leq \frac{t_0}{h^1_1}\bigg(-\frac{2}{F}h^1_kh^k_1-\frac{1}{F}\bar{R}^1_{\nu 1\nu}+\frac{g^{k1}F^{pq}}{F^2} \big(h^m_{q}\bar{R}_{kp1m}+h^m_{1}\bar{R}_{kpqm}+\nabla_p\bar{R}_{k1q\nu}+\nabla_k\bar{R}_{1pq\nu} \big)\bigg)\\\nonumber
&+\frac{t_0g^{kl}F^{pq}}{F^2u}\nabla_l\Phi\bar{R}_{\nu pkq}-\frac{ t_0F^{pq}}{F^2}\left(\phi\phi^\prime\delta_{pq}+\frac{2\phi^\prime r_pr_q}{\phi}\right)+\frac{t_0F^{pq}}{F^2}\left(\frac{h^1_{1p}h^1_{1q}}{(h^1_1)^2}-\frac{u_pu_q}{u^2} \right)+\frac{2vt_0}{F}+C
\end{align}

by Lemma \ref{ADS curvature} and Lemma \ref{C^1 estimate}, we have
\begin{align*}
&\frac{g^{k1}F^{pq}}{F^2} \big(h^m_{q}\bar{R}_{kp1m}+h^m_{1}\bar{R}_{kpqm}+\nabla_p\bar{R}_{k1q\nu}+\nabla_k\bar{R}_{1pq\nu} \big)\\
=&\frac{ F^p_p}{F^2}(-h^p_p+h^1_1)+O(e^{-\alpha t_0})\\
=&-\frac{1}{F}+\frac{ F^p_p}{F^2}h^1_1+O(e^{-\alpha t_0})
\end{align*}

similarly,
\begin{align*}
-\frac{1}{F}\bar{R}^1_{\nu 1\nu}=\frac{1}{F}+O(e^{-\alpha t_0}), \quad \frac{g^{kl}F^{pq}}{F^2u}\nabla_l\Phi\bar{R}_{\nu pkq}=O(e^{-\alpha t_0})
\end{align*}

Plug into (\ref{fffformula}), we have
\begin{align*}
0&\leq \frac{t_0}{h^1_1}\left(-\frac{2}{F}h^1_kh^k_1+\frac{F^p_p}{F^2}h^1_1\right)+\frac{2vt_0}{F}+C\\\nonumber
&-\frac{ t_0F^{pq}}{F^2}\left(\phi\phi^\prime\delta_{pq}+\frac{2\phi^\prime r_pr_q}{\phi}\right)+\frac{t_0F^{pq}}{F^2}\left(\frac{h^1_{1p}h^1_{1q}}{(h^1_1)^2}-\frac{u_pu_q}{u^2} \right)
\end{align*}

By the critical equation, we have
\begin{align*}
\frac{F^{pq}}{F^2}\left(\frac{h^1_{1p}h^1_{1q}}{(h^1_1)^2}-\frac{u_pu_q}{u^2} \right)=\frac{F^{pq}}{F^2}\left(-\frac{2u_pr_q}{u}+r_pr_q\right)
\end{align*}

Since 
\begin{align*}
\nabla_i u &=g^{kl}h_{ik}\nabla_l \Phi=g^{kl}h_{ik}\phi r_l
\end{align*}
together with Lemma \ref{C^1 estimate}, we have
\begin{align*}
\frac{t_0F^{pq}}{F^2}\left(\frac{h^1_{1p}h^1_{1q}}{(h^1_1)^2}-\frac{u_pu_q}{u^2} \right)\leq C
\end{align*}

thus
\begin{align*}
0&\leq \frac{t_0}{h^1_1}\left(-\frac{2}{F}h^1_kh^k_1+\frac{F^p_p}{F^2}h^1_1\right)+\frac{2vt_0}{F}+C\\\nonumber
&-\frac{ t_0F^{pq}}{F^2}\left(\phi\phi^\prime\delta_{pq}+\frac{2\phi^\prime r_pr_q}{\phi}\right)
\end{align*}

Again by Lemma \ref{C^1 estimate} and the relation $\phi^\prime=\phi+O(1)$, we have
\begin{align*}
0&\leq \frac{t_0}{h^1_1}\bigg(-\frac{2}{F}h^1_kh^k_1+\frac{F^{p}_p}{F^2} h^1_1\bigg)+\frac{2t_0}{F}+C-\frac{ t_0F^p_p}{F^2}\\
&= -\frac{2t_0}{F}h^1_1+\frac{2t_0}{F}+C
\end{align*}

thus
\begin{align*}
h^1_1-1\leq \frac{C}{t_0}
\end{align*}
we have
\begin{align*}
w\leq t_0\log\left(1+\frac{C}{t_0}\right)+t_0\left(-\log\tilde{u}+\tilde{r}-\log 2 \right)\leq C
\end{align*}
thus
\begin{align*}
\left(\log h^1_1-\log\tilde{u}+\tilde{r}-\log 2 \right)t\leq C
\end{align*}
for any $t$, together with (\ref{lower bound of}), we have
\begin{align*}
\limsup_{t\rightarrow \infty}\sup_{M}\kappa_i(t,\cdot)\leq 1
\end{align*}
\end{proof}

\begin{lemm}\label{lower bound of F}
$ F\geq n-Cte^{-2\alpha t}$, where C depends on $\Sigma_0,n,k$.
\end{lemm}

\begin{proof}
Consider the test function $w=\frac{v}{F}$, thus $\dot{\varphi}=\frac{1}{G}=\frac{w}{\phi}$, we have
\begin{align*}
\frac{\partial w}{\partial t}=\phi \frac{\partial \dot{\varphi}}{\partial t}+\phi\phi^\prime \dot{\varphi}^2
\end{align*}
Let $G^{ij}=\frac{\partial G}{\partial \varphi_{ij}}$, $G^k=\frac{\partial G}{\partial \varphi_k}$, then
\begin{align*}
G^{ij}=-\frac{1}{v^2}F^{i}_l\tilde{\sigma}^{lj}
\end{align*}
similar to Lemma \ref{time derivative}, we have
\begin{align*}
\frac{\partial w}{\partial t}&=\frac{\phi}{v^2G^2}\left(F^{i}_l\tilde{\sigma}^{lj}\dot{\varphi}_{ij}-v^2 G^k\dot{\varphi}_k-F^{i}_i\phi\phi^{\prime\prime} \dot{\varphi}\right)+\frac{\phi^\prime}{\phi}w^2\\
&=\frac{w^2}{v^2\phi}\left(F^{i}_l\tilde{\sigma}^{lj}\left(\frac{w}{\phi}\right)_{ij}-v^2 G^k\left(\frac{w}{\phi}\right)_k-F^{i}_i\phi^{\prime\prime}w\right)+\frac{\phi^\prime}{\phi}w^2\\
&=\frac{w^2}{v^2\phi^2}\left(F^{i}_l\tilde{\sigma}^{lj}w_{ij}-\frac{2}{\phi} F^{i}_l\tilde{\sigma}^{lj}w_i\phi_j-v^2 G^kw_k\right)\\
&+\frac{w^2}{v^2\phi^2}\left(\frac{2w}{\phi^2}F^{i}_l\tilde{\sigma}^{lj}\phi_i\phi_j-\frac{w}{\phi}F^{i}_l\tilde{\sigma}^{lj}\phi_{ij}+\frac{v^2w}{\phi}G^k\phi_k\right)\\
&+\frac{\phi^\prime}{\phi}w^2-\frac{F^{i}_i\phi^{\prime\prime}}{v^2\phi}w^3
\end{align*}

First, note that $w$ is bounded by our previous estimate, thus we only need to consider the second line.

By Lemma \ref{C^1 estimate}, We have
\begin{align*}
\frac{2w}{\phi^2}F^{i}_l\tilde{\sigma}^{lj}\phi_i\phi_j\leq C e^{(\frac{2}{n}-\alpha)t}
\end{align*}

Now by (\ref{4.3})
\begin{align*}
\phi_{ij}&=\phi^\prime r_{ij}+\phi^{\prime\prime}r_ir_j\\
&=\phi\phi^\prime \varphi_{ij}+\phi\left({\phi^\prime}^2+\phi\phi^{\prime\prime}\right)\varphi_i\varphi_j
\end{align*}
thus
\begin{align*}
-F^{i}_l\tilde{\sigma}^{lj}\phi_{ij}&=-\phi\phi^\prime F^{i}_l\tilde{\sigma}^{lj}\varphi_{ij}-\phi\left({\phi^\prime}^2+\phi\phi^{\prime\prime}\right)F^{i}_l\tilde{\sigma}^{lj}\varphi_i\varphi_j\\
&\leq -\phi\phi^\prime F^{i}_l\left(\phi^\prime \delta^l_i-\phi v h^l_i\right)+Ce^{(\frac{3}{n}-2\alpha )t}
\end{align*}
By lemma \ref{sigma_k 2}, 
\begin{align*}
-F^{i}_l\tilde{\sigma}^{lj}\phi_{ij}&\leq -\phi\phi^\prime\left(n\phi^\prime -\phi v F\right)+Ce^{(\frac{3}{n}-2\alpha )t}\\
&=\phi\phi^\prime\left(v^2\frac{\phi}{w} -n\phi^\prime \right)+Ce^{(\frac{3}{n}-2\alpha )t}
\end{align*}
i.e.
\begin{align*}
-\frac{w}{\phi}F^{i}_l\tilde{\sigma}^{lj}\phi_{ij}\leq \phi\phi^\prime v^2-n{\phi^\prime}^2w+Ce^{(\frac{2}{n}-2\alpha )t}
\end{align*}

Now, consider $G^k$, we have
\begin{align*}
G^k=\frac{F^i_k\varphi_{ij}}{v^2}\frac{\varphi^i}{v^2}-2\frac{F^i_l\varphi_{ij}}{v^2}\frac{\varphi^l\varphi^j\varphi^k}{v^4}-2\frac{F}{v^4}\varphi^k
\end{align*}
since $h^i_j$ is bounded, by (\ref{4.6}), we have $|\varphi_{ij}|\leq Ce^{\frac{t}{n}}$. thus
\begin{align*}
G^k\phi_k=\phi\phi^\prime G^k\varphi_k=\phi\phi^\prime\left(\frac{F^i_k\varphi_{ij}}{v^4}\varphi^j\varphi_k-2\frac{F^i_l\varphi_{ij}}{v^6}\varphi^l\varphi^j|\nabla\varphi|^2-2\frac{F}{v^4}|\nabla \varphi|^2\right)\leq Ce^{(\frac{3}{n}-2\alpha)t}
\end{align*}
Put all together, we have

\begin{align*}
\frac{\partial w}{\partial t}&\leq \frac{w^2}{v^2\phi^2}\left(F^{i}_l\tilde{\sigma}^{lj}w_{ij}-\frac{2}{\phi} F^{i}_l\tilde{\sigma}^{lj}w_i\phi_j-v^2 G^kw_k\right)\\
&+\frac{w^2}{v^2\phi^2}\left(\phi\phi^\prime v^2-n{\phi^\prime}^2w\right)+\frac{\phi^\prime}{\phi}w^2-\frac{F^{i}_i\phi^{\prime\prime}}{v^2\phi}w^3+Ce^{-2\alpha t}\\
&\leq \frac{w^2}{v^2\phi^2}\left(F^{i}_l\tilde{\sigma}^{lj}w_{ij}-\frac{2}{\phi} F^{i}_l\tilde{\sigma}^{lj}w_i\phi_j-v^2 G^kw_k\right)\\
&+2\frac{\phi^\prime}{\phi}w^2-n\frac{\phi\phi^{\prime\prime}+{\phi^\prime}^2}{v^2\phi^2}w^3+Ce^{-2\alpha t}
\end{align*}
by Lemma \ref{C^1 estimate} and Lemma \ref{ADS curvature}, we have
\begin{align*}
\frac{d}{dt}w_{max}\leq 2w_{max}^2-2nw_{max}^3+Ce^{-2\alpha t}
\end{align*}

By Lemma \ref{upbound of F}, we have $w_{max}\geq \frac{1}{n}$, thus
\begin{align*}
\frac{d}{dt}w_{max}\leq \frac{2}{n^2}-\frac{2}{n} w_{max}+Ce^{-2\alpha t}
\end{align*}
thus
\begin{align*}
w_{max}\leq \frac{1}{n}+Cte^{-2\alpha t}
\end{align*}
thus
\begin{align*}
F\geq n-Cte^{-2\alpha t}
\end{align*}

\end{proof}

Put lemma \ref{upbound of F} and lemma \ref{lower bound of F} together, we have
\begin{coro}\label{cor 1}
\begin{align*}
|h^i_j-\delta^i_j|\rightarrow 0,
\end{align*}
as $t\rightarrow\infty$
\end{coro}

Now let's compute the convergence rate, we have the following lemma,
\begin{lemm}
\begin{align*}
|h^i_j-\delta^i_j|\leq O(e^{-\frac{2}{n}t}).
\end{align*}
\end{lemm}
\begin{proof}
Consider the test function
\begin{align*}
G=\frac{1}{2}\sum_{ij}\left(h^i_j-\delta^i_j\right)\left(h^j_i-\delta^j_i\right)e^{\lambda t}
\end{align*}
we have
\begin{align*}
\dot{G}=\sum_{ij}\dot{h^i_j}\left(h^j_i-\delta^j_i\right)e^{\lambda t}+\lambda G
\end{align*}
for each $t$, $G$ attains maximum at some point $x_0$, at $x_0$
\begin{align*}
&\sum_{ij} h^i_{jk}\left(h^j_i-\delta^j_i\right)=0\\
&\sum_{ij} h^i_{jkl}\left(h^j_i-\delta^j_i\right)+h^i_{jk}h^j_{il}\leq 0
\end{align*}
thus
\begin{align*}
\dot{G}&= \bigg(-\frac{1}{F}h^i_kh^k_j+ \frac{F^{pq,rs}{h_{pq}}^ih_{rsj}}{F^2}  -\frac{2F^{pq}{h_{pq}}^iF^{rs}h_{rsj}}{F^3}-\frac{1}{F}\bar{R}^i_{\nu j\nu}\\\nonumber
&+\frac{g^{ki}F^{pq}}{F^2} \big( h_{kj,pq}-h^m_{q}(h_{km}h_{pj}-h_{kj}h_{mp})-h^m_{j}(h_{mk}h_{pq}-h_{kq}h_{mp})\\\nonumber
&+h^m_{q}\bar{R}_{kpjm}+h^m_{j}\bar{R}_{kpqm}+\nabla_p\bar{R}_{kjq\nu}+\nabla_k\bar{R}_{jpq\nu} \big)\bigg)\left(h^j_i-\delta^j_i\right)e^{\lambda t}+\lambda G
\end{align*}
by the critical equation, we have
\begin{align*}
\dot{G}&\leq \bigg(-\frac{1}{F}h^i_kh^k_j+ \frac{F^{pq,rs}{h_{pq}}^ih_{rsj}}{F^2}  -\frac{2F^{pq}{h_{pq}}^iF^{rs}h_{rsj}}{F^3}-\frac{1}{F}\bar{R}^i_{\nu j\nu}\\\nonumber
&+\frac{g^{ki}F^{pq}}{F^2} \big( -h^m_{q}(h_{km}h_{pj}-h_{kj}h_{mp})-h^m_{j}(h_{mk}h_{pq}-h_{kq}h_{mp})\\\nonumber
&+h^m_{q}\bar{R}_{kpjm}+h^m_{j}\bar{R}_{kpqm}+\nabla_p\bar{R}_{kjq\nu}+\nabla_k\bar{R}_{jpq\nu} \big)\bigg)\left(h^j_i-\delta^j_i\right)e^{\lambda t}-\frac{F^{pq}}{F^2}h^i_{jp}h^j_{iq}e^{\lambda t}+\lambda G
\end{align*}
by Corollary \ref{cor 1}, all the terms involving the derivatives of $h^i_j$ can be controlled by $-\frac{F^{pq}}{F^2}h^i_{jp}h^j_{iq}$, thus
\begin{align*}
\dot{G}&\leq \bigg(-\frac{1}{F}h^i_kh^k_j-\frac{1}{F}\bar{R}^i_{\nu j\nu}+\frac{ g^{ki}F^{pq}}{F^2} \big( -h^m_{q}(h_{km}h_{pj}-h_{kj}h_{mp})-h^m_{j}(h_{mk}h_{pq}-h_{kq}h_{mp})\\\nonumber
&+h^m_{q}\bar{R}_{kpjm}+h^m_{j}\bar{R}_{kpqm}+\nabla_p\bar{R}_{kjq\nu}+\nabla_k\bar{R}_{jpq\nu} \big)\bigg)\left(h^j_i-\delta^j_i\right)e^{\lambda t}+\lambda G
\end{align*}

Diagonalized it, we have
\begin{align*}
g_{ij}=\delta_{ij}, h_{ij}=\kappa_i\delta_{ij},\kappa_1\leq \cdots\leq \kappa_n
\end{align*}
and by Lemma \ref{ADS curvature}, Lemma \ref{C^1 estimate} and Lemma \ref{upbound of F}, we have
\begin{align*}
\dot{G}&\leq \bigg(-\frac{1}{F}\kappa_i^2+\frac{F^{pp}}{F^2}\bigg(\kappa_i\kappa_p^2-\kappa_i^2\kappa_p+\kappa_i\bigg)+ce^{-\frac{2}{n}t}\bigg)\left(\kappa_i-1\right)e^{\lambda t}+\lambda G\\
&=\bigg(-\frac{2}{F}\left(\kappa_i^2-\kappa_i\right)+\frac{F^{pp}}{F^2}\kappa_i\left(\kappa_p-1\right)^2+ce^{-\frac{2}{n}t}\bigg)\left(\kappa_i-1\right)e^{\lambda t}+\lambda G\\
&\leq \bigg(-\frac{4}{F}\kappa_i+\lambda+2\frac{F^{pp}}{F^2}\kappa_i|\kappa_i-1|\bigg)G+c \left(\kappa_i-1\right)e^{(-\frac{2}{n}+\lambda)t }
\end{align*}
Thus if we choose $\lambda$ small enough, we conclude that $G$ is bounded, i.e. $|h^i_j-\delta^i_j|=O(e^{-\frac{\lambda}{2} t})$ for small $\lambda$. 

Now if we choose $\tilde{G}=\sup_M\frac{1}{2}|h^i_j-\delta^i_j|^2e^{\frac{4t}{n}}$, we have
\begin{align*}
\dot{\tilde{G}}&\leq \bigg(-\frac{4}{F}\kappa_i+\frac{4}{n}+2\frac{F^{pp}}{F^2}\kappa_i|\kappa_i-1|\bigg)\tilde{G}+ce^{-\frac{\lambda t}{2}}\\
&\leq ce^{-\frac{\lambda t}{2}}\tilde{G}+ce^{-\frac{\lambda t}{2}}
\end{align*}
write $\sqrt{\tilde{G}}=f$, we have
\begin{align*}
\dot{f}\leq ce^{-\frac{\lambda t}{2}}f+ce^{-\frac{\lambda t}{2}}
\end{align*}
thus $f\leq C$, we proved the lemma.
\end{proof}

\medskip

\noindent
{\it Acknowledgement}: The author would like to express gratitude to his supervisor Professor Pengfei Guan for consistent support and encouragement.

\end{document}